\documentclass[11pt,leqno]{article}
\usepackage{graphicx}
\usepackage{amsmath,amsthm}
\usepackage{amssymb,amsfonts,amscd,mathrsfs,latexsym}
\usepackage{mathtools}
\usepackage{thmtools}
\usepackage{color}
\usepackage{xcolor}
\usepackage{cases}
\usepackage{enumitem}
\usepackage{relsize}
\usepackage{setspace}
\usepackage{accents}
\usepackage[round]{natbib}
\usepackage[colorlinks,citecolor=blue,urlcolor=blue,
linkcolor=blue]{hyperref}

\newtheorem{thrm}{Theorem}[section]
\newtheorem*{thrm*}{Theorem}
\newtheorem{lemm}[thrm]{Lemma}
\newtheorem*{lemm*}{Lemma}
\newtheorem{prop}[thrm]{Proposition}
\newtheorem*{prop*}{Proposition}

\newtheorem*{corl*}{Corollary}

\newtheorem*{claim*}{Claim}

\theoremstyle{definition}
\newtheorem{defn}[thrm]{Definition}
\newtheorem*{defn*}{Definition}

\newtheorem*{assm*}{Assumption}

\newtheorem*{exmp*}{Example}

\newtheorem*{rmrk*}{Remark}

\numberwithin{equation}{section}

\renewcommand{\qed}{\hfill\mbox{\raggedright\rule{0.1in}{0.1in}}}

\begin{document}

\title{\textbf{On the time consistency of collective preferences}}

\author{\textsc{Luis A. Alcal\'{a}}\thanks{Departamento de Matem\'{a}tica \& Instituto de Matem\'{a}tica Aplicada-San Luis (IMASL), Universidad Nacional de San Luis-CONICET, San Luis, ARGENTINA. Email: lalcala@unsl.edu.ar}
}

\date{July 15, 2018}

\maketitle

\begin{abstract}
A dynamic model of collective consumption and saving decisions made by a finite number of agents with constant but different discount rates is developed. Collective utility is a weighted sum of individual utilities with time-varying utility weights. Under standard separability assumptions, it is shown that collective preferences may be nonstationary but still satisfy time consistency. The assumption of time-varying weights is key to balance the need of the group for a changing distribution of consumption among its members over time with their tolerance for consumption fluctuations.
\medskip

\medskip

\medskip

\end{abstract}

\section{Introduction}
\label{sec:intro}

This paper develops a framework for constructing collective preferences from the Pareto optimal consumption and saving decisions of time consistent individuals with heterogeneous discount factors. One of the main issues addressed in this framework is whether the aggregation of heterogeneous time preferences necessarily implies time inconsistency. This aggregation problem has been long recognized, but it still poses challenges to researchers. For instance, \citet{dumas89} shows that difficulties in modelling aggregate behavior in the presence of heterogeneity may even arise in a two-agent setting with homogeneous discount factors. A more general approach has been taken by \citet{gollierzeckh05} in the analysis of a pure  exchange economy composed of a finite group of agents with different discount factors. They show that Paretian aggregation of individual preferences does not generally yield a constant aggregate discount rate. In fact, optimality implies that the individual shares of a common consumption stream change over time, which allows to define an aggregate discount rate. Collective impatience is a weighted average of individual discount rates, with weights that are proportional to the agents' \emph{tolerance for consumption fluctuations} (a concept similar to risk tolerance when agents have preferences over stochastic consumption streams). 

The conditions for time consistent aggregation of individual preferences are very strong, as shown by \citet{zuber11}. Aggregate preferences are stationary and Pareto optimal if and only if individual and collective preferences are additively separable with exponential discounting, and all agents have the same discount rate. There are strong implications for models of collective choice as well. For instance,  \citet{jacksonyariv14,jacksonyariv15} show that when preferences of heterogeneous agents evaluating a stream of common consumption are aggregated via a collective utility function that satisfies Pareto optimality, collective preferences must be dictatorial or time inconsistent. On the other hand, if aggregation rules are Paretian and nondictatorial, they exhibit a specific form of time inconsistency, known as present-bias or decreasing impatience. 

As is common in social choice theory, previous results are based on a fundamental assumption: aggregate utility is a weighted sum of individual utilities with weights that are \emph{constant over time}. This paper shows that relaxing this assumption makes time consistent aggregation of individuals with heterogeneous discount factors possible. To do this, the problem of constructing Pareto optimal allocations is extended to allow for time-varying weights, based on the framework developed by \citet{lucasstokey84}. Utility weights then become a state variable in the dynamic program. Important extensions introduced by \citet{danalevan90,danalevan91}, such as allowing for unbounded returns and more general specifications of the discounting process, are also included in the analysis. 

A sharp insight provided by \citet{halevy15} will also be used to characterize the collective preferences obtained in this work. Specifically, assume that preferences are represented by a sequence of preference relations $\{\succsim_t\}_{t=0}^\infty$ defined over temporal payments $(c,t)$, where $c$ is a real number and $t$ represents time. For every $t,t' \geq 0$, $b,c \in \mathbb{R}$, and $\tau,\tau' \geq 0$, three distinct properties are defined as follows:
\begin{enumerate}[label=(H\arabic*),leftmargin=*,itemsep=0pt,topsep=1pt]
\item\label{axm:station} Stationarity: $(b,t+\tau) \sim_t (c,t+\tau') \iff (b,t'+\tau) \sim_t (c,t'+\tau')$,
\item\label{axm:timeinv} Time invariance: $\ (b,t+\tau) \sim_t (c,t+\tau') \iff (b,t'+\tau) \sim_{t'} (c,t'+\tau')$,
\item\label{axm:timecon} Time consistency: $(b,t+\tau) \sim_t (c,t+\tau') \iff (b,t+\tau) \sim_{t'} (c,t+\tau')$.
\end{enumerate}
It is shown in \cite{halevy15} that any two properties imply the third. Therefore, if preferences are time invariant, they satisfy time consistency if and only if they are stationary. But the distinction between stationarity and time consistency as two separate phenomena has been often overlooked in the literature. 

Based on this fundamental distinction between time consistency, time invariance, and stationarity, \citet{millnerheal18} have recently shown that time-additive collective preferences can be both nondictatorial and time consistent. If agents have heterogeneous discount factors and utility weights are a function of time, they find a specific rule that utility weights must obey for collective preferences to be time consistent. In the present paper, we show that time consistency is not only \emph{possible} in a dynamic economy with consumption and savings, but also that the specified functional form for the utility weights can be the outcome of agents' optimizing behavior and Pareto efficient sharing rules.    

To work with an axiomatic approach and develop conditions for consistent aggregation, it is necessary to impose certain separation properties on collective preferences. In this respect, the close relation between the concepts of risk tolerance and \emph{tolerance for consumption fluctuations} is clearly beneficial. The classical results from optimal risk sharing and syndicate theory in a static equilibrium by \citet{wilson68} and \citet{amerstoeck83} are used to formulate  aggregate utility as the product of two factors: a discounting function depending only on time and the utility derived from current consumption. Along these lines, related aggregation results in pure exchange Arrow-Debreu economies have been obtained by \citet{schmedders07}, assuming risk aversion and complete markets, and \citet{wakai07}, under homogeneous ambiguity. They show that if individual utility functions satisfy the property of linear risk tolerance with identical marginal risk tolerance, a two-fund separation property holds for these economies. 

A similar approach as the one presented here to determine efficient intertemporal allocations in a multi-agent setting has been used in \citet{dumasuppalwang00}. They assume a stochastic endowment economy in continuous-time. Although their work does not touch on the issue of time consistency, an interesting aspect is that individuals have recursive utilities that are not time-additive, allowing to distinguish risk aversion from  intertemporal substitution. They also obtain an aggregation result for the case that each agent has Kreps-Porteus preferences with identical risk aversion.

The recent literature in time preferences and intertemporal choice is too extensive to be reviewed here in any detail.\footnote{See the comprehensive survey by \citet{fredloewodon02} and references therein. A detailed historical account can also be found in the work of \citet{palhuerta03}.} Besides the studies that have already been cited, recent works more closely related to this paper include \citet{banermitra07} and \citet{banerdubey13} on impatience as a general property of Paretian social welfare functions. Relations between stationarity, decreasing impatience, and time consistency have been studied by \citet{prelec04}, \cite{bleichetal09}, and \citet{rohde09}. Characterization and welfare properties of competitive equilibria with time-dependent preferences are given in \citet{herinrohde06}, \citet{luttmariott06}, and \citet{dziewulski15}.\footnote{Other recent contributions on related topics include: \citet{drugeon00}, \citet{das03}, \citet{hayashi03}, \citet{benoitok07}, \citet{okmasatl07}, \citet{mutlu13}, \citet{depazetal13}, and \citet{montolea14}.}

The rest of the paper is organized as follows. The next section develops the model and provides the main result for aggregation. Section \ref{sec:collective} gives a characterization of collective preferences and optimal sharing rules. Additional properties are discussed in Section \ref{sec:discussion}, where the result on time consistent aggregation is proved. Concluding remarks are offered in Section \ref{sec:conclude}. Technical proofs are relegated to the Appendix.

\section{Setup and Aggregation}
\label{sec:setup}

\subsection{Preliminaries}
\label{sec:prelims}

Let $\mathbb{R}$ denote the set of real numbers. Then, $\mathbb{R}_+\!:=[0,+\infty)$ and $\underline{\mathbb{R}}:= \mathbb{R} \cup \{-\infty\}$. The set of all nonnegative integers is represented by $\mathbb{Z}_+$. Vectors in the $n$-dimensional Euclidean space $\mathbb{R}_+^n$ are denoted by $x$. For infinite sequences of these vectors, the notation used is $\mathbf{x}:=(x_0,x_1,\ldots)$. The space of real-valued sequences $\ell^\infty$ is endowed with the sup norm, $\Vert \mathbf{x}\Vert_\infty:=\sup_t|x_t|$. If $n > 1$ and each component is in $\ell^\infty$, the corresponding product space will be denoted as $(\ell^\infty)^n$ (for ease of notation, the infinity superscript will be dropped if no ambiguity arises). 

The following conventions, which are commonly used in convex optimization problems, are also adopted:
\begin{enumerate}[label=(C\arabic*),leftmargin=*,itemsep=0pt,topsep=1pt]
\item $0\cdot(\pm\infty)=(\pm\infty)\cdot 0 = 0$;
\item $a\cdot(\pm\infty)=(\pm\infty)\cdot a = \pm\infty$, for all $a \in (0,+\infty]$;
\item $a\cdot(\pm\infty)=(\pm\infty)\cdot a = \mp\infty$, for all $a \in [-\infty,0)$.
\end{enumerate}

\subsection{The model}
\label{sec:model}

For some integer $n \geq 2$, a set $N=\{1,\ldots,n\}$ of infinitely lived agents make collective consumption and saving decisions. Time is discrete and denoted by $t=0,1,2,\ldots$ There is a single consumption good which can either be consumed or transformed one-to-one into capital available next period. Production occurs in a single unit where agents pool their capital holdings to obtain a certain deterministic return in terms of the consumption good, which is called ``consumption'' for simplicity. But the framework is quite general and admits various interpretations. 

Time preferences are represented by an additively separable intertemporal utility function with instantaneous utility $u$,  and constant geometric discounting. Instantaneous utility is  assumed to be common to all agents, but discount factors $\delta^i$ are heterogeneous and satisfy
\begin{equation*}
1 > \delta^1 > \delta^2 \geq \cdots \geq \delta^n > 0,
\end{equation*}
so there are at least two separate groups of agents in terms of their impatience. 

Denote by $x_t^i \in \mathbb{R}_+$ the quantity which agent $i$ consumes at date $t$. Let $\mathbf{x}^i=(x_0^i,x_1^i,\ldots)$ denote an infinite sequence of consumptions for agent $i$. Each agent $i$ assigns an utility value to a consumption path $\mathbf{x}^i$ in the space $\ell_+$ of non-negative bounded sequences  
\begin{align*}
w_0^i(\mathbf{x}^i):=\sum_{t=0}^\infty (\delta^i)^t u_i(x_t^i), & & i=1,\ldots,n. 
\end{align*}
The following assumptions on the primitives of the problem, preferences and technology, will be used throughout the paper.
\begin{enumerate}[label=(U\arabic*),leftmargin=*,itemsep=0pt,topsep=1pt]
\item\label{asm:u0} For each $i \in N$, the instantaneous utility function $u_i: \mathbb{R}_+\!\to \underline{\mathbb{R}}$ is continuous, strictly increasing, and strictly concave. At the origin, either $u_i(0)=0$ or $u_i(0)=-\infty$.
\item\label{asm:uprime} For each $i \in N$, the instantaneous utility function $u_i$ is twice continuously differentiable on $\mathbb{R}_+ \backslash \{0\}$. If $u_i(0)=0$, then $\lim_{x \to 0^+} u_i'(x)=+\infty$.
\end{enumerate}
Technology is given by a standard neoclassical production function that transforms aggregate capital into output. 
\begin{enumerate}[label=(T\arabic*),leftmargin=*,itemsep=0pt,topsep=1pt]
\item\label{asm:f0} The production function $f:\mathbb{R}_+\!\to \mathbb{R}_+$ is continuous, strictly increasing, strictly concave, and $f(0)=0$.
\item\label{asm:fprime} The production function $f$ is differentiable on $\mathbb{R}_+\!\backslash\{0\}$, $\lim_{k \to 0^+} f'(k) > (1/\delta^n)$, and $\lim_{k \to +\infty} f'(k)=0$.
\end{enumerate}

Note that \ref{asm:f0}--\ref{asm:fprime} imply the existence of a maximum level of sustainable capital $k_m > 0$, so there is no loss of generality in restricting the state space to the closed interval $K:=[0,k_m]$. This implies that the set of feasible consumption (and utilities) is bounded above. Moreover, \ref{asm:uprime}--\ref{asm:fprime} guarantee that equilibrium paths are interior.

\begin{defn}
For each $t$, an \emph{allocation} is a pair $\left(x_t,k_{t+1}\right)$, consisting on a consumption profile $x_t:=(x_t^1,\ldots,x_t^n) \in \mathbb{R}_+^n$ and aggregate capital for the next period $k_{t+1} \in K$. Given $k_0 \geq 0$, a capital path $\mathbf{k}:=(k_0,k_1,\ldots)$ is said to be \emph{feasible} if each element of the sequence belongs to the set
\begin{align*}
\Gamma(k_t):=\{k_{t+1} \in K: 0 \leq k_{t+1} \leq f(k_t)\},
\end{align*}
which will be called the \emph{feasible correspondence} at time $t$. 
\end{defn}

\begin{defn}
The set of all \emph{feasible capital paths} from $k_0$ is defined as
\begin{align*}
\Pi(k_0):=\left\{\mathbf{k} \in \ell_+: k_{t+1} \in \Gamma(k_t),\ t=0,1,\ldots;\ k_0\ \text{given}\right\}.
\end{align*}
\end{defn}

\begin{defn}
Let $\mathbf{x}:=(x_0,x_1,\ldots)$ denote a consumption path. The set of all \emph{feasible consumption paths} from $k_0$ is defined as
\begin{align*}
\Omega(k_0):=\left\{\mathbf{x} \in \ell_+^n: 0 \leq \textstyle\sum_i x_t^i \leq f(k_t)\!-\!k_{t+1},\ t=0,1,\ldots;\ 
\text{for some}\ \mathbf{k} \in \Pi(k_0)\right\}.
\end{align*}
\end{defn}

\subsection{Optimal allocations with heterogeneous discounting}
\label{sec:pareto}

The methodology for constructing Pareto optimal allocations   is mainly based on the works of \cite{lucasstokey84} and  \cite{danalevan90,danalevan91}. Although an important difference is that individual preferences in the current framework are a primitive of the problem, instead of being derived from a time aggregator. 

A fundamental assumption in this setup is that the Pareto or utility weights for each agent are allowed to vary over time. For each period $t$, the set of weights is given by a vector $\theta_t:=(\theta_t^1,\ldots,\theta_t^n)$ in the $n$-dimensional simplex, denoted by $\Theta^n$ and defined as
\begin{equation}
\label{eq:simplex}
\Theta^n:=\left\{\theta \in \mathbb{R}_+^n : \theta^i \geq 0,\ i=1,\ldots,n;\ \text{and}\ \textstyle\sum_i\theta^i =1\right\}.
\end{equation}

The utility possibility set $\mathcal{U}(k)$ contains all the possible combinations of utility available to the $n$ agents when the initial capital stock is $k \geq 0$,
\begin{align*}
\mathcal{U}(k):=\left\{z \in \underline{\mathbb{R}}^n : z^i=w^i(\mathbf{x}^i),\ i=1,\ldots,n,\ \text{for some}\ \mathbf{x} \in \Omega(k)\right\}.
\end{align*}
And the value function of the Pareto problem is defined as the support function of the utility possibility set, that is,
\begin{equation*}
V(k,\theta):=\sup_{z \in \mathcal{U}(k)}\ \sum_{i=1}^n \theta^i z^i.
\end{equation*}

A characterization of the value function is given in the following proposition. For detailed proofs, see 
\citet[Theorem 3]{lucasstokey84} and Section 4 in \citet{danalevan90}. Also, \cite{duranlevan00} contains results that are particularly useful for the case where returns are unbounded. 

\begin{prop}
\label{prp:Vfunct}
The value function $V: K \times \Theta^n \to \underline{\mathbb{R}}$ satisfies the following properties: 
\begin{enumerate}[label=\emph{(\alph*)},align=left,leftmargin=*,
itemsep=0pt,topsep=1pt]
\item $V$ is continuous on $K\!\times\Theta^n$;  
\item For each $\theta \in \Theta^n$, $V(\cdot,\theta): K \to \mathbb{R}$ is strictly increasing and strictly concave;
\item For each $k \in K$, $V(k,\cdot): \Theta^n \to \mathbb{R}$ is homogeneous of degree one and strictly convex;
\item $V$ is continuously differentiable in the interior of $K\!\times\Theta^n$.
\end{enumerate}
\end{prop}

The next result, which completely characterizes Pareto optimal allocations, summarizes Propositions 3.1--3.2 and Theorem 3.1 in \cite{danalevan90}, and is stated below for convenience.

\begin{prop}
\label{prp:Paretochar}
Under assumptions \ref{asm:u0}--\ref{asm:uprime} on preferences and \ref{asm:f0}--\ref{asm:fprime} on technology, a feasible consumption path $\mathbf{x} \in \Omega(k_0)$ is Pareto optimal if and only if there exist sequences $\mathbf{k}$ and $\mathbf{z}$ such that $k_{t+1} \in \Gamma(k_t)$ and $z_t \in \tilde{\mathcal{U}}(k_t)$ for all $t$, where $\tilde{\mathcal{U}}$ is the frontier of $\mathcal{U}$, and $z_t^i = u_i(x_t^i)+ \delta^i z_{t+1}^i$ for all $i \in N$ and for all $t$. Moreover, $\mathbf{k}$ and $\mathbf{z}$ are uniquely determined.\footnote{If $S$ is a closed set, and $\overline{S}$ denotes its closure, the \emph{topological frontier} $\tilde{S}$ of $S$ is defined as $\overline{S} \backslash S$.} 
\end{prop}

In light of the above, the problem of finding optimal allocations can be formulated as follows. For each $(k,\theta)$ in $K \times \Theta^n$, a nonnegative consumption profile $x:=(x^1,\ldots,x^n)$, next-period capital $y$, continuation utilities $z$, and next-period Pareto weights $\tau$ are chosen to solve the following program:
\begin{align}
\label{eq:PP}
\adjustlimits\sup_{x \in X,\,y \geq 0,\,z\in \mathcal{U}\ }\inf_{\tau \in \Theta^n}\ &\sum_{i=1}^n\theta^i\left[u_i(x^i) + \delta^i z^i\right],\tag{PP}\\[5pt]
\text{s.t. } \quad & \sum_{i=1}^n x^i + y \leq f(k),\nonumber\\
&\sum_{i=1}^n \tau^i z^i - V(y,\tau) \leq 0.\nonumber
\end{align}

In what follows, the following notation for \emph{aggregate variables} is adopted:
\begin{align*}
\hat{x}:=\sum_{i=1}^n x^i, \quad \hat{x}_t:= \sum_{i=1}^n x_t^i, \quad \text{and} \quad \mathbf{\hat{x}}:=\sum_{i=1}^n \mathbf{x}^i,
\end{align*}
with $\hat{X} \subset \mathbb{R}_+$ denoting the space of feasible aggregate consumption levels. It is also convenient to define the set of all \emph{feasible aggregate consumption paths} from $k_0$, separately from $\Omega$, which is denoted as $\Psi(k_0)$ and defined by
\begin{align*}
\Psi(k_0):=\left\{\mathbf{\hat{x}} \in \ell_+: 0 \leq \hat{x}_t \leq f(k_t)-k_{t+1},\ t=0,1,\ldots;\  \text{for some}\  \mathbf{k} \in \Pi(k_0)\right\}.
\end{align*}

Based on the formulation given in \eqref{eq:PP}, the following theorem lays the foundations for the preference aggregation process and subsequent results. The proof can be found in Appendix \ref{app:pfRFBellman}.

\begin{thrm}
\label{thm:RFBellman}
Suppose that $(k,\theta) \gg 0$. Let $s:\hat{X} \times \Theta^n \to X$ be an optimal  consumption profile for \eqref{eq:PP}, as a function of aggregate consumption $\hat{x}$ and the Pareto weights $\theta$, such that $0 < \hat{x} \leq f(k)-y$, for some $y \in \Gamma(k)\backslash\{0\}$, and $\hat{x}:=\sum_i s^i(\hat{x},\theta)$. Let $F:\Theta^n \to \Theta^n$ be a transition map for the Pareto weights associated to the optimal choice of $\tau$ in \eqref{eq:PP}. Then, there exists an aggregate periodic utility function $U:\hat{X}\!\times \Theta^n \to \underline{\mathbb{R}}$, defined as
\begin{equation}
\label{eq:Udef}
U(\hat{x},\theta):=\sum_{i=1}^n \theta^i u_i\left(s^i(\hat{x},\theta)\right),
\end{equation}
and an aggregate discount rate $\mu:\Theta^n \to \mathbb{R}_+$, given by
\begin{equation}
\label{eq:mudef}
\mu(\theta):=\sum_{i=1}^n \theta^i \delta^i,
\end{equation}
such that the value of \eqref{eq:PP} satisfies the following functional equation
\begin{equation}
\label{eq:RFBellman}
V(k,\theta)=\sup_{y \in \Gamma(k)} \Big[ U(f(k)-y,\theta)+\mu(\theta)V(y,F(\theta))\Big],
\end{equation}
for all $(k,\theta)$ in the interior of $K \times \Theta^n$.
\end{thrm}

A collective periodic utility function  $U(\hat{x}_t,\theta_t)$ is constructed by aggregating individual \emph{indirect utility functions}, which in turn are the values of the following auxiliary program to \eqref{eq:PP},
\begin{align}
\label{eq:PPx}
\sup_{x \in X}\ \left\{\textstyle\sum_i\theta^i\,u_i(x^i): \textstyle\sum_i x^i \leq \hat{x}\right\},\tag{PPx}
\end{align}
where $\hat{x}$ satisfies all the hypotheses of Theorem \ref{thm:RFBellman}. This is derived from a standard result, \emph{Fisher's separation theorem}, which implies that an optimal consumption profile in \eqref{eq:PP} can be obtained separately from the investment decision, i.e., the choice of $y$. A key assumption for the separation result to hold in this case is the additivity of intertemporal preferences.

After solving \eqref{eq:PPx}, it is easy to see that   aggregate intertemporal preferences have a recursive  representation. For concreteness, let $\lambda$ and $\mu$ be the Lagrange multipliers associated with the inequality restrictions in \eqref{eq:PP}, and consider the first-order conditions of this problem.\footnote{The first-order conditions for an interior optimum are obtained by partial differentiation are
\begin{align*}
0&=\theta_t^i u_i'(x_t^i)-\lambda_t, & t=0,1,\ldots,\\
0&=\theta_t^i\delta^i-\mu_t\theta_{t+1}^i, & t=0,1,\ldots,\\
0&=z_t^i - V_{\theta^i}(k_{t+1},\theta_{t+1}), & t=0,1,\ldots,\\
0&=-\lambda_t + \mu_t V_k(k_{t+1},\theta_{t+1}), & t=0,1,\ldots,
\end{align*}
together with transversality conditions for the state variables.} The conditions with respect to $\theta_{t+1}^i$ and $z_{t+1}^i$, respectively, imply that
\begin{align}
\label{eq:foctautp}
z_t^i &= u_i(x_t^i) + \delta^i z_{t+1}^i,  & t=0,1,\ldots,\\
\label{eq:focztp}
\theta_t^i \delta^i &= \mu_t \theta_{t+1}^i, & t=0,1,\ldots.
\end{align} 
Combining \eqref{eq:foctautp} with \eqref{eq:focztp} yields
\begin{align*}
\theta_t^i z_t^i &= \theta_t^i u_i(x_t^i) + \mu_t \theta_{t+1}^i z_{t+1}^i, & & t=0,1,\ldots.
\end{align*}
Therefore, adding up both sides of the equality over $i \in N$ in the expression above,
\begin{align}
\label{eq:Wtrec}
W_t = U(\hat{x}_t,\theta_t)+\mu(\theta_t) W_{t+1}, & & t=0,1,\ldots,
\end{align}
where $W_t:=\sum_i \theta_t^i w_t^i$ is an aggregate utility index, $U(\hat{x}_t,\theta_t)$ is the value of \eqref{eq:PPx}, and $\mu(\theta_t)$, which is defined by \eqref{eq:mudef}, is obtained from adding both sides of \eqref{eq:focztp} over $i \in N$.

Now it is apparent that, taken as a function of $\theta_t$, this multiplier can be interpreted as an aggregate \emph{factor of time preference} and is related to the one-period discount rate.\footnote{The \emph{instantaneous rate of time preference} is defined as $\rho(\cdot):=1/\mu(\cdot) -1$.} In fact, it is clear from \eqref{eq:mudef}, that this aggregate factor of time preference is the weighted arithmetic mean of individual discount factors, with weights given by $\theta_t \in \Theta^n$. Moreover, the evolution of $\mu(\theta_t)$ over time is entirely explained by the evolution of the utility weights. Therefore, \eqref{eq:Wtrec} defines recursively an intertemporal utility function that assigns to each infinite stream of aggregate consumption and Pareto weights $(\mathbf{\hat{x}},\mathbf{\theta})$ in $\ell_+\!\times (\Theta^n)^\infty$ an utility measure for the collective formed by the members of $N$. This representation of collective preferences also defines an aggregate \emph{discount factor}
\begin{align*}
\beta_t:= \prod_{s=0}^{t-1} \mu(\theta_s), & & t=0,1,\ldots
\end{align*}

It may seem odd to have preferences that depend, directly or indirectly, on a vector of utility weights $\theta_t$. \citet{lucasstokey84} give Lagrange multipliers $\lambda$ and $\mu$ an interpretation of competitive equilibrium prices as functions of the state of the system $(k,\theta)$. Let $r(k,\theta)$ be the price of the consumption good in the following period in terms of the current consumption good. In a stationary equilibrium, $\mu$ is considered an equilibrium interest factor which satisfies $\mu = 1/(1+r)$. Unfortunately, this interpretation relies on the assumption of \emph{increasing marginal impatience}, which has been criticized on empirical and theoretical grounds. In the current   setup, taking $\mu$ as a subjective aggregate discount rate implies decreasing marginal impatience, as will be shown in Section \ref{sec:discussion}. 

Alternatively, the recursive representation \eqref{eq:Wtrec} can be rationalized using the concept of \emph{variational utility} introduced by \citet{geoffard96}. The variational utility of a consumption path is defined as the minimum value of an additive criterion, taken over all possible future discount rates. Note the close relation between this idea and the formulation of the Pareto problem given in \eqref{eq:supdec1} and the auxiliary program \eqref{eq:supdec2}. Variational utilities include time additive and recursive preferences as special cases. Henceforth, the model developed in this paper can be understood in the context of a broader class of deterministic models where  optimal allocations can be characterized as the solution of a simple dynamic programming problem. The value function represents aggregate utility and belongs to the class of variational utilities. A stochastic formulation of this approach can be found in \citet{dumasuppalwang00}.

From \eqref{eq:mudef} and \eqref{eq:focztp}, the dynamics of utility weights are determined by the following updating process
\begin{align}
\label{eq:thetaupdate}
\theta_{t+1}^i = \frac{\theta_t^i\delta^i}{\sum_j \theta_t^j\delta^j}, & & i=1,\ldots,n;\quad  t=0,1,\ldots,
\end{align}
which defines the transition map $F$. This is [...] 

Note that the updating process for the utility weights defined by \eqref{eq:thetaupdate} is quite relentless: if $\theta_t^i=0$ for some $i$ at any period $t$, then it will take a zero value forever, and agent $i$ is practically removed from the group. In this paper, this possibility is ruled out in equilibrium by assumptions \ref{asm:uprime} and \ref{asm:fprime}, as long as $\theta_0^i > 0$ for all $i \in N$. Another consequence of the dynamics implied by \eqref{eq:thetaupdate} is that for any $j=2,\ldots,n$, it follows that
\begin{equation*}
\frac{\theta_{t+1}^j}{\theta_{t+1}^1} =\left(\frac{\delta^j}{\delta^1}\right)\frac{\theta_t^j}{\theta_t^1} \to 0, \quad \text{as}\ t \to +\infty,
\end{equation*}
so the relative weight on aggregate utility vanishes asymptotically for all agents, with the exception of the most patient. This and \eqref{eq:simplex} immediately imply that $\theta_t^j \to 0^+$, for all $j=2,\ldots,n$, and $\theta_t^1 \to 1^-$ as $t$ approaches infinity. Hence agent 1's consumption equals $x_t$, while the others consume zero (or a minimum subsistence level), a result known as \emph{Ramsey's conjecture}. But the analysis focuses on the case where $\theta_t^i > 0$ holds for every $i \in N$ and for any finite $t$, so if the equilibrium implied by Ramsey's conjecture holds, it may hold only asymptotically.

From the above discussion, it can be concluded that \eqref{eq:PP} admits a sequential formulation with aggregate recursive preferences, which will be called the \emph{recursive preference formulation} (RPF), and has the form
\begin{align}
\label{eq:RPF}
\sup_{\mathbf{\hat{x}} \in \Psi(k_0),\,\mathbf{k} \in \Pi(k_0)}\ &\ \sum_{t=0}^\infty \beta_t\,U(\hat{x}_t,\theta_t)\tag{RPF}\\[5pt]
\text{s.t. } \quad 
& \hat{x}_t + k_{t+1} \leq f(k_t), & t=0,1,\ldots,\nonumber\\[5pt]
&\beta_{t+1} \leq \mu(\theta_t)\beta_t, & t=0,1,\ldots, \nonumber\\[5pt]
&\theta_{t+1} = F(\theta_t), & t=0,1,\ldots,\nonumber\\[5pt]
&k_0,\theta_0 > 0\ \text{given},\ \beta_0 = 1.\nonumber
\end{align} 

Furthermore, the value function $V$ associated with \eqref{eq:RPF} satisfies a generalized Bellman equation given by \eqref{eq:RFBellman}.

\section{Collective Preferences}
\label{sec:collective}

This section develops sufficient conditions under which the optimal allocations can be written as a function of the state of the system $(\hat{x}_t,\theta_t)$, and the collective utility function $U(\hat{x}_t,\theta_t)$ is multiplicatively separable (up to an additive constant). This is explained in part by the fact that the literature on dynamic choice generally assumes a preference relation $\succsim_t$ over dated commodities $(\hat{x},t)$ that can be represented by a separable function, e.g., where the discount factor, say $D(t)$, depends on time, and instantaneous utility $U(\hat{x})$ is a function of current consumption. 

In order to characterize optimal intertemporal allocations, it is useful to define an index of \emph{tolerance for consumption fluctuations} (TCF) given by $\alpha:=-u'/u''$, which has been introduced by \cite{gollierzeckh05}. Using a specific form for the instantaneous utility function $u$, that yields a TCF index as an affine function of consumption, allows to obtain the desired separability for $U$. This is hardly surprising, given the analogous relation between the concavity of the instantaneous utility function and risk tolerance/TCF.

\subsection{Separability}
\label{sec:separab}

Time additivity, as shown in the proof of Theorem \ref{thm:RFBellman}, has an important consequence, namely that the allocation of aggregate consumption among agents can be solved independently from the collective investment decision. Hence, the distributive problem can be treated as an \emph{intratemporal} or static program. This is key to obtain the separable form for $U$. 

Fix $t \geq 0$ together with a feasible level of aggregate consumption $\hat{x}_t > 0$. Given $\theta_0 > 0$, the vector of utility weights is determined by the formula $\theta_t=F^t(\theta_0)$, where $F^t$ is the $t$-fold application of the map $F$ described in \eqref{eq:thetaupdate}. Then, the collective instantaneous utility function $U:\hat{X}\!\times \Theta^n \to \underline{\mathbb{R}}$ is the value of the following program
\begin{align}
\label{eq:staticp}
\sup_{x_t \in X}\ &\sum_{i=1}^n \theta_t^i u_i(x_t^i)\tag{S}\\
\text{s.t. } \ & \sum_{i=1}^n x_t^i \leq \hat{x}_t.\nonumber\
\end{align}

\begin{defn}
A \emph{sharing rule} is a map $s: \hat{X}\!\times \Theta^n \to \mathbb{R}_+^n$ that satifies
\begin{equation}
\label{eq:resrecs}
\sum_{i=1}^n s^i(\hat{x}_t,\theta_t) = \hat{x}_t.
\end{equation}
A \emph{Pareto optimal sharing rule} is a sharing rule that solves \eqref{eq:staticp}.
\end{defn} 

The set $\mathcal{S}$ of all sharing rules $s$ is nonempty, since the allocation $s^i=\hat{x}_t/n$ for all $i \in N$ is always feasible. Note that there is a correspondence between this sharing rule and the one defined in  \eqref{eq:mguit}, given by $s^i(\hat{x},\theta)=\sigma^i(\theta^i,\lambda(\hat{x},\theta))$. The following result characterizes aggregate instantaneous utility and the optimal sharing rule.
\begin{prop}
\label{prp:Uchar}
The aggregate instantaneous utility function $U: \hat{X}\!\times \Theta^n \to \underline{\mathbb{R}}$ and the Pareto optimal sharing rule $s:\hat{X}\!\times \Theta^n \to \mathbb{R}_+^n$ satisfy the following properties: 
\begin{enumerate}[label=\emph{(\alph*)},ref=(\alph*),leftmargin=*,itemsep=0pt,
topsep=1pt]
\item\label{prp:Umonccv} For each $\theta \in \Theta^n$, $U(\cdot,\theta): \hat{X} \to \mathbb{R}$ is strictly increasing and strictly concave; 
\item\label{prp:Uhomog} For each $\hat{x} \in \hat{X}$, $U(\hat{x},\cdot):\Theta^n \to \mathbb{R}$ is homogeneous of degree one and strictly convex;
\item\label{prp:Udiffb} $U$ is twice continuously differentiable in the interior of $\hat{X}\!\times\Theta^n$;
\item\label{prp:shomog} For each $i \in N$: $s^i(0,\theta)=0$,  for all $\theta \in \Theta^n$; and $s^i(\hat{x},\cdot):\Theta^n \to \mathbb{R}_+^n$ is homogeneous of degree zero, for all $\hat{x} \in \hat{X}$;
\item\label{prp:sdiffb} $s$ is continuously differentiable in the interior of $\hat{X}\!\times\Theta^n$.
\end{enumerate}
\end{prop}

The definitions of individual and aggregate indices of TFC are  given next.

\begin{defn}
The \emph{individual index of TCF} for agent $i \in N$ is a map $\alpha^i: \hat{X}\!\times \Theta^n \to \mathbb{R}_+$ defined by
\begin{equation*}
\alpha^i(\hat{x}_t,\theta_t):=-\frac{u_i'(s^i(\hat{x}_t,\theta_t))}{u_i''(s^i(\hat{x}_t,\theta_t))}.
\end{equation*}
The \emph{collective index of TCF}, $\alpha: \hat{X}\!\times\Theta^n \to \mathbb{R}_+$ is defined in terms of $U(\hat{x}_t,\theta_t)$ as 
\begin{equation*}
\alpha(\hat{x}_t,\theta_t):=-\frac{\partial{U(\hat{x}_t,\theta_t)}}{\partial{\hat{x}}}\left[\frac{\partial{U^2(\hat{x}_t,\theta_t)}}{\partial{\hat{x}^2}}\right]^{-1}.
\end{equation*}
\end{defn}

A simple characterization of individual and collective indices of TCF is given next. A similar result is obtained in  \cite{gollierzeckh05}, when $\hat{x}_t$ follows an exogenous process. Note that the result is independent of  functional forms.  

\begin{prop} 
\label{prp:alphachar}
Assume that $s \in \mathcal{S}$ is an optimal sharing rule. Then, for each $i \in N$ and for all $(\hat{x}_t,\theta_t) \in \hat{X}\!\times \Theta^n$, individual and  aggregate indices of TCF satisfy the following properties
\begin{equation*}
\alpha^i(\hat{x}_t,\theta_t)=\frac{\partial{s^i}(\hat{x}_t,\theta_t)}{\partial{\hat{x}}}\left(\sum_{j=1}^n\alpha^j(\hat{x}_t,\theta_t)\right), \quad \text{and} \quad
\alpha(\hat{x}_t,\theta_t)=\sum_{i=1}^n\alpha^i(\hat{x}_t,\theta_t).
\end{equation*}
\end{prop} 

\begin{proof}
See \cite{gollierzeckh05} and \cite{wilson68}. 
\end{proof}

In order to obtain the desired separability property for $U$, some additional results are borrowed from works on optimal risk sharing and syndicate theory. In particular, \citet{amerstoeck83} established a sufficient condition for aggregation that requires separability, based on initial findings by \citet{wilson68}. The condition is related to a class of utility functions representing individual preferences that yield affine sharing rules.\footnote{In an environment with risk averse agents, interpreting $\theta \in \Theta^n$ as the realization of some random variable with a known distribution, if individual utility functions belong to the hyperbolic absolute risk aversion (HARA) class with identical cautiousness, optimal sharing rules are affine functions.} The same class of utility functions  satisfies an analogous property in terms of TCF.
\begin{defn}
A periodic utility function $u:\mathbb{R}_+\!\to \underline{\mathbb{R}}$  satisfies the property of \emph{affine tolerance for consumption fluctuations} (ATCF) if it has the form
\begin{align}
\label{eq:ATCFdef}
u(x)=
\begin{cases}
\frac{\gamma}{1-\gamma}\left[\left(\phi + \frac{\eta}{\gamma}\,x\right)^{1-\gamma}-1\right], &  0 < \gamma < +\infty,\ \gamma \neq 1,\\[5pt]
\log(\phi+\eta\,x), & \gamma =1,
\end{cases}
\end{align}
with $\phi+(\eta/\gamma)\,x \geq 0$, $0 < \eta < +\infty$, and $\phi \in \mathbb{R}$. This family of parametric utility functions will be referred to as the \emph{ATCF class}.
\end{defn} 

Note that the ATCF class includes utility functions that are commonly used in the literature. For instance,  $\phi =0$ and $\eta=1$ yields the power utility function. As $\gamma \to 1$, it converges to the logarithmic form $u(x)=\log\left(\phi+\eta\,x\right)$, with $\phi + \eta\,x >0$. For $\phi=1$, taking the limit $\gamma \to +\infty$ yields the exponential form  $u(x)=1-\exp(-\eta\,x)$. Note that if $\phi < 0$, consumption has a lower bound, i.e., $x \geq -\phi\gamma/\eta$, then $u$ takes the form of Stone-Geary preferences. The case $\phi \geq 0$ is compatible with $x \geq 0$.

Suppose that each individual utility belongs to the ATFC class. In order to obtain exact aggregation, which allows for analytical results, the following particular form of an individual  utility function is adopted
\begin{equation}
\label{eq:uatfcdef}
u_i(x_t^i):=\frac{\gamma}{1-\gamma}\left[\left(\phi^i + \frac{\eta}{\gamma}\,x_t^i\right)^{1-\gamma}-1\right],
\end{equation}
with $\phi^i \in \mathbb{R}$ for each $i \in N$. Exact aggregation mainly depends on the assumption that $\gamma^i = \gamma$, for all $i \in N$. Adding heterogeneity through the parameter $\eta$ does not affect the results, but complicates the calculations in a substantive manner. For the remaining of this section, the ATFC class for individual utilities is always assumed to have the form given in \eqref{eq:uatfcdef}. The following result characterizes optimal sharing rules under ATCF preferences.


\begin{prop}
\label{prp:sharerule}
Assume that each agent has an instantaneous utility function $u_i$ that belongs to the ATCF class. Further assume $(k_0,\theta_0) \gg 0$ and that the solution to \eqref{eq:staticp} is interior for every $t$. Then, the optimal sharing rule for each agent $i \in N$ has the form
\begin{align}
\label{eq:sharerule}
s^i(\hat{x}_t,\theta_t)=a^i(\theta_t)\,\hat{x}_t + b^i(\theta_t),
\end{align}
where $a^i(\theta_t) \geq 0$, 
$\sum_i a^i(\theta_t)=1$, and $ 
\sum_i b^i(\theta_t)=0$, for all $\theta_t \in \Theta^n$.
\end{prop}

\begin{proof}
For each $i$ and for each $t$, the optimal sharing rule is given by
\begin{align*}
\sigma^i(\theta_t^i,\lambda_t)=\frac{\gamma}{\eta}\left[\left(\frac{\eta\theta_t^i}{\lambda_t}\right)^{\frac{1}{\gamma}}-\phi^i\right].
\end{align*}
Adding up over $i$ yields aggregate consumption $\hat{x}_t$ in terms of $(\theta_t,\lambda_t)$,
\begin{align*}
\hat{x}_t = \sum_{i=1}^n \sigma^i(\theta_t^i,\lambda_t)=\frac{\gamma}{\eta}\left[\left(\frac{\eta}{\lambda_t}\right)^{\frac{1}{\gamma}}\sum_{i=1}^n(\theta_t^i)^{\frac{1}{\gamma}}-\hat{\phi}\right],
\end{align*}
where $\hat{\phi}:= \sum_{i=1}^n \phi^i$. Rearranging terms in the above expression, 
\begin{align*}
\lambda_t= \left(\sum_{i=1}^n(\theta_t^i)^{\frac{1}{\gamma}}\right)^\gamma \eta \left(\hat{\phi}+\frac{\eta}{\gamma}\,\hat{x}_t \right)^{-\gamma}, 
\end{align*}
and after appropriate substitutions, 
\begin{align*}
s^i(\hat{x}_t,\theta_t) &= \left[\frac{(\theta_t^i)^{\frac{1}{\gamma}}}{\sum_j(\theta_t^j)^{\frac{1}{\gamma}}}\right]\,\hat{x}_t + \frac{\gamma\hat{\phi}}{\eta}\left[\frac{(\theta_t^i)^{\frac{1}{\gamma}}}{\sum_j(\theta_t^j)^{\frac{1}{\gamma}}}-\frac{\phi^i}{\hat{\phi}}\right],
\end{align*}
hence the sharing rule is an affine function of $\hat{x}_t$, with coefficients given by 
\begin{align}
\label{eq:coeffshare}
a^i(\theta_t)=\frac{(\theta_t^i)^{\frac{1}{\gamma}}}{\sum_j(\theta_t^j)^{\frac{1}{\gamma}}} \quad \text{and} \quad
b^i(\theta_t)=\frac{\gamma\hat{\phi}}{\eta}\left[\frac{(\theta_t^i)^{\frac{1}{\gamma}}}{\sum_j(\theta_t^j)^{\frac{1}{\gamma}}}-\frac{\phi^i}{\hat{\phi}}\right],
\end{align}
as claimed. 
\end{proof}

Motivated by the characterization given in Proposition \ref{prp:sharerule}, certain efficiency and fairness criteria implied by the optimal sharing rule can be identified. Let $\tilde{\theta}_t^i$ denote the \emph{effective Pareto weight} for agent $i$, defined as
\begin{align}
\label{eq:thetahat}
\tilde{\theta}_t^i:=\frac{(\theta_t^i)^{\frac{1}{\gamma}}}{\sum_j(\theta_t^j)^{\frac{1}{\gamma}}}, & & t=0,1,\ldots,
\end{align}
and note that $0 \leq \tilde{\theta}_t^i \leq 1$ and $\sum_i \tilde{\theta}_t^i =1$.
In other words, $\tilde{\theta}_t:=\big(\tilde{\theta}_t^1,\ldots,\tilde{\theta}_t^n\big) \in \Theta^n$. Since the optimal sharing rule given in \eqref{eq:sharerule} is homogeneous of degree zero, it can be alternatively written as
\begin{align}
\label{eq:srulehat}
s^i(\hat{x}_t,\tilde{\theta}_t) &= \tilde{\theta}_t^i\,\hat{x}_t +\frac{\gamma\hat{\phi}}{\eta}\left(\tilde{\theta}_t^i-\tilde{\phi}^i\right), & & i=1,\ldots,n;\quad t=0,1,\ldots,
\end{align}
where $\tilde{\phi}^i:=\phi^i/\hat{\phi}$. The first term on the right-hand side of \eqref{eq:srulehat} implies that each agent $i$ receives a proportion $\tilde{\theta}_t^i$ of aggregate consumption, which can be interpreted as an efficiency criterion. The second term is proportional to the difference between the corresponding effective weight $\tilde{\theta}^i$ and the individual parameter $\tilde{\phi}^i$, establishing a fairness criterion. 

It is clear from \eqref{eq:srulehat} that the optimal sharing rule $s^i$ allocates aggregate consumption according to the effective weights and a system of transfers between agents. This transfer scheme is purely redistributive, since aggregate transfers add up to zero. Loosely speaking, $s^i$ shows an implicit individual  ``preference for income redistribution'' determined by $\tilde{\phi}^i$. 

For simplicity, assume that $\phi^i = \phi$ for all $i \in N$.   To rule out uninteresting cases, also assume that $\tilde{\theta}_t^i \neq \tilde{\phi}^i=(1/n)$ holds for at least one $i \in N$.  Three distinct cases  emerge:
\begin{enumerate}[label=(\roman*),leftmargin=*]
\item If $\phi=0$, then $s^i$ implies a \emph{neutral} preference for income redistribution. This is the case when $u$ is a power function or has the log form.
\item If $\phi < 0$, then $s^i$ implies a preference for \emph{top-to-bottom} income redistribution. The set of agents $N$ is partitioned into a subset $D \subset N$ of agents in a relative disadvantaged situation, 
\begin{align*}
D:=\left\{j \in N: \left(\tilde{\theta}_t^j - \frac{1}{n}\right) < 0\right\},
\end{align*}
who receive a transfer from those in a position of relative advantage grouped in a subset $A \subset N$,
\begin{align*}
A:=\left\{l \in N:\left(\tilde{\theta}_t^l - \frac{1}{n}\right) \geq 0\right\}.
\end{align*}
Naturally, the sets $D$ and $A$ are disjoint and $N = D \cup A$. 
\item If $\phi >0$, then $s^i$ implies a preference for \emph{bottom-to-top} income redistribution, so this case is  opposite to the previous case.
\end{enumerate}

Simple substitution shows that the collective instantaneous utility function, when individual utilities belong to the ATCF class, has the form
\begin{equation}
\label{eq:Ucthetadef}
U(\hat{x}_t,\theta_t)=\frac{\gamma}{1-\gamma}\left[\left(\sum_i (\theta_t^i)^{\frac{1}{\gamma}}\right)^\gamma\left(\hat{\phi} +\frac{\eta}{\gamma}\,\hat{x}_t\right)^{1-\gamma}-1\right],
\end{equation}
which is clearly separable in $x_t$ and $\theta_t$ (up to an additive constant). Note also that the aggregate TCF index is independent of $\theta_t$ and each individual TCF index is separable in $(\hat{x}_t,\theta_t)$, that is,
\begin{equation*}
\hat{\alpha}(\hat{x}_t,\theta_t)=\frac{\hat{\phi}}{\eta}+\frac{1}{\gamma}\,\hat{x}_t \quad \text{and} \quad
\alpha^i(\hat{x}_t,\theta_t)=a^i(\theta_t)\,\hat{\alpha}(\hat{x}_t),
\end{equation*}
where $a^i(\theta_t)$ is defined in \eqref{eq:coeffshare}.  Moreover, $\partial{s^i}/\partial{\hat{x}} = a^i(\theta_t)$, for all  $i \in N$ and $\theta \in \Theta^n$. It is then verified that both $U$ and $s$ satisfy the properties of Propositions \ref{prp:Uchar} and \ref{prp:alphachar}.

\subsection{Equivalence Results}
\label{sec:equiv}

This section presents several equivalent formulations of the Pareto problem \eqref{eq:PP}. For this, some of the previous  results require additional specifications. Let $\mu:\Theta^n \to [0,1]$ represent the aggregate discount rate defined as 
\begin{align}
\label{eq:mutdef}
\mu(\theta_t) := \sum_{i=1}^n \theta_t^i\delta^i, & & t=0,1,\ldots
\end{align}
Note that this function takes values on $[\delta_n,\delta_1] \subset [0,1]$ for all $\theta_t \in \Theta^n$, it is strictly  increasing in each argument, and homogeneous of degree one. 

Since $\mu$ is bounded on $\Theta^n$, it is clear that the sequence of discount factors $\{\beta_t\}_{t=0}^\infty$ is bounded on the space $(0,1)^\infty$. This, together with \eqref{eq:thetaupdate}, implies that the program \eqref{eq:RPF} is well-behaved, in a sense that will become clear below, and it satisfies a generalized Bellman equation.\footnote{The existence of a solution to functional equations arising in dynamic programming with generalized discounting has been shown under fairly general conditions by \citet{bhaktamitra84} and  \cite{bhaktachoud88} for the case of bounded returns, and by  \cite{jaskiewiczetal14} for unbounded returns.} For $\hat{x}_t > 0$ and $\theta_t$ in the interior of $\Theta^n$, the Euler equation associated with \eqref{eq:RPF} is given by
\begin{align}
\label{eq:RPFEuler}
\frac{\partial{U}(\hat{x}_t,\theta_t)}{\partial{\hat{x}}}=\mu(\theta_t)\,\frac{\partial{U}(\hat{x}_{t+1},\theta_{t+1})}{\partial{\hat{x}}}\,f'(k_{t+1}), 
\end{align}
and assuming ATCF preferences, it follows from \eqref{eq:Ucthetadef} and \eqref{eq:mutdef} that
\begin{align}
\label{eq:RPFEuler1}
\left(\textstyle\sum_i(\theta_t^i)^{\frac{1}{\gamma}}\right)^\gamma
\left(\hat{\phi} +\tfrac{\eta}{\gamma}\,\hat{x}_t\right)^{-\gamma}\!= \left(\textstyle\sum_i \theta_t^i\delta^i\right)\left(\textstyle\sum_i(\theta_{t+1}^i)^{\frac{1}{\gamma}}\right)^\gamma
\left(\hat{\phi} +\tfrac{\eta}{\gamma}\,\hat{x}_{t+1}\right)^{-\gamma}  f'(k_{t+1}).
\end{align}

Given that condition \eqref{eq:RPFEuler} is necessary and sufficient for optimality, it is possible to define an aggregate utility function $\hat{U}$, which is independent of $\theta_t$, and an effective discount rate $\hat{\mu}$ satisfying an equivalent condition, i.e.,
\begin{align}
\label{eq:RPFEuler2}
\tilde{U}'(\hat{x}_t) = \tilde{\mu}(\theta_t)\tilde{U}'(\hat{x}_{t+1})f'(k_{t+1}).
\end{align}
It is apparent from \eqref{eq:RPFEuler1} that $\tilde{U}$ belongs to the ATCF class, as can be seen by substituting $\phi^i$ for $\hat{\phi}$ and $x_t^i$ for $\hat{x}_t$ in \eqref{eq:ATCFdef}. Therefore, 
\begin{equation}
\label{eq:RPFUhatdef}
\tilde{U}(\hat{x}_t)=\frac{\gamma}{1-\gamma}\left[\left(\hat{\phi} +\frac{\eta}{\gamma}\,\hat{x}_t\right)^{1-\gamma}-1\right].
\end{equation}
Collecting the terms that depend on $\theta_t$ and $\theta_{t+1}$ in \eqref{eq:RPFEuler1} and using \eqref{eq:thetaupdate},  the \emph{effective aggregate discount rate} $\tilde{\mu}(\theta)$ is defined as
\begin{align}
\label{eq:muhatdef}
\tilde{\mu}(\theta_t):=\left[\frac{\sum_i (\theta_t^i\delta^i)^{\frac{1}{\gamma}}}{\sum_i (\theta_t^i)^{\frac{1}{\gamma}}}\right]^\gamma=\left[\textstyle\sum_i a^i(\theta_t^i)\,(\delta^i)^{\frac{1}{\gamma}}\right]^\gamma.
\end{align}
As $\mu$ is a weighted arithmetic mean, the effective discount rate $\tilde{\mu}$ is a \emph{weighted generalized mean} (or \emph{weighted power mean}) of the individual discount factors with weights $a^i(\theta_t^i)$, for each $i \in N$, and exponent $\gamma$. Only if individual preferences $u_i$ are logarithmic, $\mu$ and $\tilde{\mu}$ coincide.

Hence there is a dynamic program equivalent to \eqref{eq:RPF}. Since $\tilde{U}$ is differentiable, the solution satisfies an Euler equation equivalent to \eqref{eq:RPFEuler1}, which is given by \eqref{eq:RPFEuler2}. By Theorem \ref{thm:RFBellman}, the equivalent program also satisfies \eqref{eq:RFBellman}. These results are summarized in the next proposition.  

\begin{prop}
\label{prp:RPFprime}
Assume that preferences are ATCF. Then, the following program is equivalent to \eqref{eq:RPF}
\begin{align}
\label{eq:RPFprime}
\sup_{\mathbf{\hat{x}}\in \Psi(k_0),\,\mathbf{k} \in \Pi(k_0)}\ &\ \sum_{t=0}^\infty \tilde{\beta}_t\,\tilde{U}(\hat{x}_t)\tag{RPF'}\\
\text{s.t. } \quad & \hat{x}_t + k_{t+1} \leq f(k_t), & & t=0,1,\ldots,\nonumber\\[5pt]
&\tilde{\beta}_{t+1} \leq \tilde{\mu}(\theta_t)\tilde{\beta}_t, & & t=0,1,\ldots,\nonumber\\[5pt]
&\theta_{t+1}=F(\theta_t), & & t=0,1,\ldots,\nonumber\\[5pt]
&k_0,\theta_0 > 0\ \text{given},\ \hat{\beta}_0 = 1.\nonumber
\end{align} 
Moreover, \eqref{eq:RPFprime} satisfies the following Bellman equation for each $(k,\theta) \in K\!\times\Theta^n$,
\begin{equation*}
\tilde{V}(k,\theta)= \sup_{y \in \Gamma(k)}\Big[\tilde{U}(f(k)-y)+\tilde{\mu}(\theta)\,\tilde{V}(y,F(\theta))\Big],
\end{equation*}
where $\tilde{U}$ and $\tilde{\mu}$ are given by \eqref{eq:RPFUhatdef} and \eqref{eq:muhatdef}, respectively.
\end{prop}

The last step to obtain separable preferences from \eqref{eq:RPFprime} is completed as follows. It is clear from \eqref{eq:thetaupdate} that $\theta_t$ can be explicitly solved as a function of $t$ and, in turn, the aggregate discount rate  $\mu_t$ will depend on $t$ directly and not through $\theta_t$. This leads to a \emph{nonstationary dynamic programming problem}, an approach that has been recently studied by \citet{kamihigashi08}, and it involves an aggregate utility function $U: \hat{X}\!\times \mathbb{Z}_+ \to \underline{\mathbb{R}}$ and a value function $V: K \times \mathbb{Z}_+ \to \underline{\mathbb{R}}$ that depend directly on time.   

Given the initial weights $\theta_0 \in \Theta^n$, solving recursively for $\theta_t^i$ in \eqref{eq:thetaupdate} yields
\begin{align}
\label{eq:thetanonst} 
\theta_t^i = \frac{\theta_0^i(\delta^i)^t}{\sum_j \theta_0^j(\delta^j)^t}, & & t=0,1,\ldots,
\end{align}
which immediately implies
\begin{align}
\label{eq:munonst}
\mu_t = \frac{\sum_i\theta_0^i(\delta^i)^{t+1}}{\sum_i \theta_0^i(\delta^i)^t}, & & t=0,1,\ldots.
\end{align}
Next, set $\beta_0 = 1$ so each member of the sequence $\{\beta_t\}_{t=0}^\infty$ is given by
\begin{align}
\label{eq:betatdef}
\beta_t=\prod_{s=0}^{t-1} \mu_s = \sum_{i=1}^n \theta_0^i (\delta^i)^t, \qquad \text{for all }  t \geq 1.
\end{align}

Now suppose that $\theta_0^i > 0$ for each $i \in N$. The following  nonstationary formulation (NSF) of the Pareto problem is equivalent to \eqref{eq:RPF} and has the form
\begin{align}
\label{eq:NSF}
\sup_{\mathbf{\hat{x}} \in \Psi(k_0),\,\mathbf{k} \in \Pi(k_0)}\ &\ \liminf_{T \to \infty}\ \sum_{t=0}^T \beta_t\,U_t\left(\hat{x}_t\right) \tag{NSF}\\[5pt]
\text{s.t. } \quad
& \hat{x}_t + k_{t+1} \leq f(k_t), & t=0,1,\ldots,\nonumber\\[5pt]
& k_0 > 0\ \text{given}.\nonumber
\end{align} 
In addition, the value function $V_t$ associated to \eqref{eq:NSF} satisfies the (modified) Bellman equation
\begin{align*}
V_t(k_t)= &\sup_{k_{t+1} \in \hat{\Gamma}(k_t)} \Big[U_t(f(k_t)-k_{t+1}) + \beta_{t+1}V_{t+1}(k_{t+1})\Big],
\end{align*}
where $\beta_t$ is given by \eqref{eq:betatdef} for each $t$,  and $\hat{\Gamma}(k_t)$ is defined as
\begin{equation}
\label{eq:Gammahatdef}
\hat{\Gamma}(k_t):= \Big\{k_{t+1} \in \Gamma(k_t): U_t(f(k_t)-k_{t+1}) > -\infty\Big\}.
\end{equation}
This result is a straightforward application of Theorem 1 in \cite{kamihigashi08}.\footnote{It can be shown that a Bellman equation holds in a strict sense iff there is no $k_{t+1} \in \Gamma(k_t)$ such that $U_t(f(k_t)-k_{t+1})=-\infty$ and $V_{t+1}(k_{t+1})=+\infty$ simultaneously. See \cite[Theorem 2]{kamihigashi08}.}

Note that the previous formulation does not assume specific functional forms. Under ATCF preferences, collective instantaneous utility for the NSF is obtained substituting \eqref{eq:thetanonst} and \eqref{eq:munonst} into \eqref{eq:Ucthetadef}, which yields
\begin{equation*}
U_t(\hat{x})=\frac{\gamma}{1-\gamma}\left[\frac{\left(\sum_i (\theta_0^i(\delta^i)^t)^{\frac{1}{\gamma}}\right)^\gamma}{\sum_i\theta_0^i(\delta^i)^t}\left(\hat{\phi} +\frac{\eta}{\gamma}\,\hat{x}\right)^{1-\gamma}-1\right].
\end{equation*}
This function satisfies the following asymptotic properties that considerably  simplify the formulation of the nonstationary program,
\begin{align*}
U_0(\hat{x}):=\lim_{t \to 0^+} U_t(\hat{x})=\frac{\gamma}{1-\gamma}\left[\left(\textstyle\sum_i(\theta_0^i)^{\frac{1}{\gamma}}\right)^\gamma\left(\hat{\phi} +\frac{\eta}{\gamma}\,\hat{x}\right)^{1-\gamma}-1\right], \quad \text{for all}\ \hat{x} \in \hat{X},
\end{align*}
and
\begin{align*}
\lim_{t \to +\infty} U_t(\hat{x})=\frac{\gamma}{1-\gamma}\left[\left(\hat{\phi} +\frac{\eta}{\gamma}\,\hat{x}\right)^{1-\gamma}-1\right]=U(\hat{x}), \qquad \text{for all}\ \hat{x} \in \hat{X}.
\end{align*}
Moreover, the sequence $\{U_t(\hat{x})\}_{t=0}^\infty$ is strictly decreasing if $0 < \gamma \leq 1$ (resp. strictly increasing if $1 < \gamma < +\infty$) for any $\hat{x} \in \hat{X}$, and takes values on the closed interval $[U(\hat{x}),U_0(\hat{x})]$ (resp. $[U_0(\hat{x}),U(\hat{x})]$).   

Recall that $\tilde{\theta}_0 \in \Theta^n$ denotes the vector of  \emph{effective initial Pareto weights}, which is obtained from \eqref{eq:thetahat} for $t=0$,
\begin{align*}
\tilde{\theta}_0^i=\frac{(\theta_0^i)^{\frac{1}{\gamma}}}{\sum_j(\theta_0^j)^{\frac{1}{\gamma}}}, & & i = 1,\ldots,n.
\end{align*}
For ease of notation, define the \emph{effective individual discount factor} as $\tilde{\delta}^i:=(\delta^i)^{1/\gamma}$, $i \in N$, so the \emph{effective discount factor} $\tilde{\beta}_t$ associated with \eqref{eq:NSF} is given by
\begin{align}
\label{eq:betathat}
\tilde{\beta}_t=\left(\sum_{i=1}^n\tilde{\theta}_0^i(\tilde{\delta}^i)^t\right)^\gamma, & & t=0,1,\ldots
\end{align}
In a similar way as in the previous case, the Euler equation associated with \eqref{eq:NSF} is equivalent to \eqref{eq:RPFEuler1} with instantaneous utility $\tilde{U}(\hat{x})$ and effective discount factor $\tilde{\beta}_t$,
\begin{align*}
\tilde{\beta}_t\,\tilde{U}'(\hat{x}_t) &=\tilde{\beta}_{t+1}\,\tilde{U}'(\hat{x}_{t+1})f'(k_{t+1}),
\end{align*}
therefore
\begin{align*}
\left(\sum_{i=1}^n\tilde{\theta}_0^i(\tilde{\delta}^i)^t\right)^\gamma\left(\hat{\phi} +\frac{\eta}{\gamma}\,\hat{x}_t\right)^{-\gamma}=\left(\sum_{i=1}^n\tilde{\theta}_0^i(\tilde{\delta}^i)^{t+1}\right)^\gamma\left(\hat{\phi} +\frac{\eta}{\gamma}\,\hat{x}_{t+1}\right)^{-\gamma}f'(k_{t+1}).
\end{align*}
The next proposition, analogous to Proposition \ref{prp:RPFprime}, summarizes the results for the NSF.

\begin{prop}
\label{prp:NSFprime}
Assume that preferences are ATCF. Then, there exists an equivalent program to \eqref{eq:NSF} which has the form
\begin{align}
\label{eq:NSFprime}
\sup_{\mathbf{\hat{x}} \in \Psi(k_0),\,\mathbf{k} \in \Pi(k_0)}\ &\ \sum_{t=0}^\infty \tilde{\beta}_t\,\tilde{U}(\hat{x}_t)\tag{NSF'}\\[5pt]
\text{s.t. } \quad &\hat{x}_t + k_{t+1} \leq f(k_t),&  & t=0,1,\ldots,\nonumber\\[5pt]
&k_0 > 0\ \text{given},\nonumber
\end{align} 
and satisfies the (modified) Bellman equation
\begin{align*}
\tilde{V}_t(k_t)= \sup_{k_{t+1} \in \hat{\Gamma}(k_t)}\Big[\tilde{U}(f(k_t)-k_{t+1})+\tilde{\beta}_{t+1}\,\tilde{V}_{t+1}(k_{t+1})\Big],& & t=0,1,\ldots,
\end{align*}
where $\hat{\Gamma}$ has been defined in \eqref{eq:Gammahatdef} and $\tilde{\beta}_t$ is given by \eqref{eq:betathat}, for all $t$.
\end{prop}

Having completed the construction of collective preferences of the sought form for $U$, the following section offers a characterization of these preferences.

\section{Discussion}
\label{sec:discussion}

Several properties of collective intertemporal utility functions are discussed in this section. First, preferences are characterized in terms of impatience. In particular, it is shown that they satisfy decreasing marginal impatience. Next, the properties of stationarity, time invariance, and time consistency --introduced in Section \ref{sec:intro}-- will be verified for these  collective preferences. To conclude, a comparison of the model developed in the previous section is made with a formulation of the Pareto problem under constant utility weights and an endogenous consumption-saving decision.

\subsection{Collective impatience}

\begin{defn}
Given an aggregate instantaneous utility function $U(\hat{x}_t,\theta_t)$ and a discount factor $\beta_t$, \emph{the marginal rate of intertemporal substitution} between present and future aggregate consumption, $\hat{x}_t$ and $\hat{x}_{t+1}$, is defined as
\begin{align*}
\mathcal{M}[(\hat{x}_t,\hat{x}_{t+1});(\theta_t,\theta_{t+1})]:=\beta_t\,\frac{\partial{U(\hat{x}_t,\theta_t)}}{\partial{\hat{x}}}\left[\beta_{t+1}\frac{\partial{U(\hat{x}_{t+1},\theta_{t+1})}}{\partial{\hat{x}}}\right]^{-1}.
\end{align*}
\end{defn}

Impatience is typically measured as the \emph{pure rate of time preference} $\hat{\rho}_t$, which in turn is defined as the marginal rate of intertemporal substitution when $\hat{x}_t=\hat{x}_{t+1}=\hat{x}$, for some $\hat{x} \in \hat{X}$. If the function $U$ belongs to the ATCF class, as in \eqref{eq:Ucthetadef}, it follows that
\begin{align*}
(1+ \hat{\rho}_t):=\frac{\left(\sum_i (\theta_t^i)^{\frac{1}{\gamma}}\right)^\gamma\left(\hat{\phi} +\frac{\eta}{\gamma}\,\hat{x}\right)^{-\gamma}}{\left(\sum_i \delta^i\theta_t^i\right)\left(\sum_i(\theta_{t+1}^i)^{\frac{1}{\gamma}}\right)^\gamma\left(\hat{\phi} +\frac{\eta}{\gamma}\,\hat{x}\right)^{-\gamma}}=\left[\frac{\sum_i(\theta_t^i)^{\frac{1}{\gamma}}}{\sum_i(\theta_t^i\delta^i)^{\frac{1}{\gamma}}}\right]^\gamma,
\end{align*}
and the discount rate implied by this expression is exactly $ \tilde{\rho}_t=1/\tilde{\mu}(\theta_t)-1$, where $\tilde{\mu}(\cdot)$ is given by \eqref{eq:muhatdef}. But this yields the same   effective discount factor as in \eqref{eq:betathat}. Then, the above expression offers an alternative way to obtain the effective aggregate discount factor which is independent of the optimization process.

The following lemma is needed for future reference, but it is  also an interesting result on its own, since it characterizes the behavior of aggregate discount factors as a function of time. 
\begin{lemm}
\label{lem:betadiff}
Assume that $\beta_t$ is defined as in \eqref{eq:betatdef} for all $t$. Then, for every $0 \leq t \leq t'$ and $0 \leq \tau \leq \tau'$,\footnote{This can be assumed without loss of generality, given that it is always possible to relabel $t,t'$ and $\tau,\tau'$ in order to satisfy the required inequalities.}
\begin{align*}
\frac{\beta_{t+\tau}\hfill}{\beta_{t+\tau'}}-\frac{\beta_{t'+\tau}\hfill}{\beta_{t'+\tau'}} \geq 0,
\end{align*}
with strict inequality if $t < t'$ and $\tau < \tau'$.
\end{lemm}

Given that $\beta_t$ is decreasing in $t$, there is impatience at the aggregate level. In addition, Lemma \ref{lem:betadiff} establishes that collective preferences, as constructed in the previous section, satisfy the property of \emph{diminishing marginal impatience}, in a sense consistent with a discrete-time formulation.

\subsection{Stationarity, time invariance, and time consistency} 

In terms of the three axioms introduced by \cite{halevy15}, it is shown below that the dynamic choice model  developed in Section \ref{sec:collective} yields collective preferences that are both nonstationary and time-dependent. At the same time, and perhaps surprisingly, these preferences satisfy time consistency.

Before presenting the main result of this section, note that discount factors $\beta_t$ (or, equivalently, $\tilde{\beta}_t$) transform period $t$ values into period 0 units. Given the multiplicative nature of discount factors derived from recursive preferences, it is easy to see that for any $t,\tau \geq 0$, discounting from $(t+\tau)$ to $t$ is equivalent to multiplication by the factor $\beta_{t+\tau}/\beta_t$.

\begin{thrm}
\label{thm:3props}
Suppose that a sequence of collective preference relations over $(\hat{x},t)$ is represented by a separable utility function of the form $\tilde{\beta}_t\,\tilde{U}(\hat{x})$, where $\tilde{U}$ and $\tilde{\beta}_t$ are given by \eqref{eq:RPFUhatdef} and  \eqref{eq:betathat}, respectively. Furthermore, assume $\theta_0^i > 0$ for each $i \in N$. Then, collective time  preferences satisfy: 
\begin{enumerate}[align=left,leftmargin=*,itemsep=0pt,topsep=1pt]
\item[$\neg\,$\ref{axm:station}] nonstationarity, 
\item[$\neg\,$\ref{axm:timeinv}] time dependency,
\item[\ \ \ref{axm:timecon}] time consistency.
\end{enumerate}
\end{thrm}

\begin{proof}
First, note that the conclusion of Lemma \ref{lem:betadiff} remains valid if $\beta_t$ is replaced by $\tilde{\beta}_t$. The reason is because substituting $\theta_0^i$ with $\tilde{\theta}_0^i$ and $\delta^i$ with $\tilde{\delta}^i$ in the definition of $\beta_t$ yields $\sum_i\tilde{\theta}_0^i(\tilde{\delta}^i)^t$. It is easy to see from \eqref{eq:betathat} that $\tilde{\beta}_t$ is an order-preserving transformation of $\beta_t$ with respect to $t$. Then, the Lemma also holds for this monotonic transformation. 

Now assume that $b,c \in \hat{X}$, $0 \leq t < t'$ and $0 \leq \tau < \tau'$. From \ref{axm:station}, one of the indifference conditions for stationarity (on the left-hand side) implies that
\begin{equation}
\label{eq:1stcond}
\frac{\tilde{\beta}_{t+\tau}}{\tilde{\beta}_t}\,\tilde{U}(b)=\frac{\tilde{\beta}_{t+\tau'}}{\tilde{\beta}_t}\,\tilde{U}(c).
\end{equation}
Hence, the following equality (on the right-hand side) must hold, 
\begin{equation*}
\frac{\tilde{\beta}_{t+\tau}}{\tilde{\beta}_t}\frac{\tilde{\beta}_{t'+\tau}}{\tilde{\beta}_{t+\tau}}\,\tilde{U}(b)=\frac{\tilde{\beta}_{t+\tau'}}{\tilde{\beta}_t}\frac{\tilde{\beta}_{t'+\tau'}}{\tilde{\beta}_{t+\tau'}}\,\tilde{U}(c),
\end{equation*}
which is equivalent to
\begin{equation}
\label{eq:indiff1}
\frac{\tilde{\beta}_{t+\tau}}{\tilde{\beta}_{t+\tau'}}-\frac{\tilde{\beta}_{t'+\tau}}{\tilde{\beta}_{t'+\tau'}}=0.
\end{equation}

To verify \ref{axm:timeinv}, suppose that \eqref{eq:1stcond} holds. Time invariance requires that, from the perspective of $t'$ preferences, applying the same time delays ($\tau$ and $\tau'$) to each alternative preserves indifference, that is,   
\begin{equation}
\label{eq:indiff2}
\frac{\tilde{\beta}_{t'+\tau}}{\tilde{\beta}_t'}\,\tilde{U}(b)=\frac{\tilde{\beta}_{t'+\tau'}}{\tilde{\beta}_t'}\,\tilde{U}(c).
\end{equation}
Combine \eqref{eq:1stcond} and \eqref{eq:indiff2} to obtain 
\begin{align*}
\frac{\tilde{\beta}_{t+\tau}}{\tilde{\beta}_{t'+\tau}}-\frac{\tilde{\beta}_{t+\tau'}}{\tilde{\beta}_{t'+\tau'}}
=\frac{\tilde{\beta}_{t+\tau'}}{\tilde{\beta}_{t'+\tau}}\left(\frac{\tilde{\beta}_{t+\tau}}{\tilde{\beta}_{t+\tau'}}-\frac{\tilde{\beta}_{t'+\tau}}{\tilde{\beta}_{t'+\tau'}}\right)=0,
\end{align*}
which reduces to \eqref{eq:indiff1}. In other words, the same  condition must be satisfied for \ref{axm:station} and \ref{axm:timeinv} to hold. But Lemma \ref{lem:betadiff} implies that if $t < t'$ and $\tau < \tau'$, then
\begin{equation*}
\frac{\tilde{\beta}_{t+\tau}}{\tilde{\beta}_{t+\tau'}}-\frac{\tilde{\beta}_{t'+\tau}}{\tilde{\beta}_{t'+\tau'}} > 0,
\end{equation*}
which contradicts \eqref{eq:indiff1}. This proves $\neg$\ref{axm:station} and $\neg$\ref{axm:timeinv}.

From \ref{axm:timecon}, time consistency implies that the following two conditions should be satisfied simultaneously,
\begin{equation*}
\frac{\tilde{\beta}_{t+\tau}}{\tilde{\beta}_t}\,\tilde{U}(b)=\frac{\tilde{\beta}_{t+\tau'}}{\tilde{\beta}_t}\,\tilde{U}(c) \quad \text{and} \quad \frac{\tilde{\beta}_{t+\tau}}{\tilde{\beta}_{t'}}\,\tilde{U}(b)=\frac{\tilde{\beta}_{t+\tau'}}{\tilde{\beta}_{t'}}\,\tilde{U}(c).
\end{equation*}
Since $\tilde{\beta}_t$ and $\tilde{\beta}_{t'}$ are strictly positive by hypothesis, it immediately follows that collective  preferences are time consistent and the proof is complete.
\end{proof}

Note that the proof of Theorem \ref{thm:3props} captures an interesting  aspect of the analysis carried out by  \cite{halevy15}. Given that any two properties imply the third, if one property is satisfied, say, time consistency, and another one is not, e.g., stationarity, the third property must not be satisfied. This is the reason why the conditions for stationarity and time invariance are identical.

\subsection{Heterogeneous discounting with constant Pareto  weights}

To close this section, the results obtained from the current framework will be compared with those of  \cite{jacksonyariv14,jacksonyariv15} to offer some insights into their results and evaluate possible extensions to this paper. The authors show that under heterogeneous discounting, collective preferences that satisfy Pareto optimality must be either time inconsistent or dictatorial. But they assume constant utility weights and a common consumption stream that is exogenous. Consequently, key instruments to resolve intertemporal conflicts are missing. 

For simplicity, suppose that $u$ is bounded on $\mathbb{R}_+$. Let $\bar{\theta} \in \Theta^n$ be a vector of constant (time homogeneous) weights, and add to \eqref{eq:PP} the restriction $\theta_t^i = \theta_{t+1}^i = \bar{\theta}^i$, for all $i$ and for all $t$. Let $J(\cdot,\bar{\theta})$ denote  the corresponding value function. Then, the right-hand side of \eqref{eq:longLagr} is replaced by
\begin{align*}
\adjustlimits\sup_{(x,y,z,\lambda)\in\Phi\,}\inf_{\mu \geq 0}\Big\{\textstyle\sum_i \bar{\theta}^i\left[u_i(x^i)-\lambda\textstyle\sum_i x^i\right] + \lambda[f(k)-y]+\textstyle\sum_i\bar{\theta}^i\left(\delta^i -\mu\right)z^i + \mu J(y,\bar{\theta})\Big\}.
\end{align*}
Given that the choice of $\hat{x}=\sum_i x^i$ can still be separated from the investment decision and the resource restriction holds with equality, following similar arguments as in the proof of Theorem \ref{thm:RFBellman}, the above problem is reduced to
\begin{align}
\label{eq:J0}
\sup_{z \in \mathcal{U}}\adjustlimits\sup_{y \in \Gamma(k)\,}\inf_{\mu \geq 0}\Big\{U(f(k)-y,\bar{\theta})+\left[\textstyle\sum_i\bar{\theta}^i\delta^i z^i - \mu\textstyle\sum_i \bar{\theta}^i z^i\right] + \mu J(y,\bar{\theta})\Big\}.
\end{align}
By hypothesis, each $z^i$ must be chosen from some interval $Z^i:=[0,z_m^i]$ with $z_m^i > 0$. 

Consider, for instance, an ``egalitarian solution'' to \eqref{eq:J0} in which $z^i = z > 0$ and $z \in \bigcap_i Z^i$ for all $i \in N$. Hence, $J(y,\bar{\theta})=z$. This implies that
\begin{align*}
J(k,\bar{\theta})=\sup_{y \in \Gamma(k)}\Big\{U(f(k)-y,\bar{\theta}) + \bar{\delta} J(y,\bar{\theta})\Big\},
\end{align*}
where $\bar{\delta}:=(\textstyle\sum_i\bar{\theta}^i \delta^i)$, so the aggregate discount factor is the weighted average of all individual discount factors. This formulation can be interpreted as the dynamic program of a fictional ``representative agent'' with an average discount factor and time consistent preferences. But these preferences will not satisfy unanimity, unless all the $\bar{\theta}^i$ are identical (i.e., equal to $(1/n)$ for each $i$).

Another possibility is to consider an ``efficient solution'' that consists in setting $\mu$ so that $\bar{\theta}^i(\delta^i-\mu) =0$, the first-order condition from differentiating the objective with respect to each $z^i$. However, this condition may be satisfied only \emph{for a single} $i \in N$. Which value of $\delta^i$ should be chosen? Suppose that $\mu=\delta^1$. It follows that $z^j =0$ for all $j > 1$, hence the problem reduces to
\begin{align}
\label{eq:J1}
J(k,\bar{\theta})=\sup_{y \in \Gamma(k)}\Big\{U(f(k)-y,\bar{\theta}) + \delta^1 J(y,\bar{\theta})\Big\}.
\end{align}
But this is equivalent to finding the optimum for $i=1$ and leaving the other agents with zero utility, thus collective preferences will be dictatorial. Assume that $\mu=\delta^i$ for some $i$ in $N\backslash\{1\}$. In that case, $z_k =0$ for all $k > i$, then the problem has the form  
\begin{align}
\label{eq:J2}
\adjustlimits\sup_{y \in \Gamma(k)\,}\sup_{(z^1,\ldots,z^i)\in  Z^1 \times \cdots \times Z^i}\Big\{U(f(k)-y,\bar{\theta})+\sum_{j < i} \bar{\theta}^j(\delta^j -\delta^i)z^j + \delta^i J(y,\bar{\theta})\Big\}.
\end{align}
Note that the marginal contribution of each agent $1 \leq j < i$ to the aggregate continuation value is precisely $\bar{\theta}^j(\delta^j-\delta^i)$, hence the choice of $(z^1,\ldots,z^i)$ can be formulated as a linear programming problem, i.e.,
\begin{equation}
\label{eq:J3}
\sup_{(z^1,\ldots,z^i) \in  Z^1 \times \cdots \times Z^i}\ \sum_{j < i} \bar{\theta}^j(\delta^j -\delta^i)z^j.
\end{equation}
If the solution lies on a vertex of the convex polytope that describes the feasible region, at most two agents get nonzero utility. If it lies on an edge or a face of the polytope, two or more agents obtain positive continuation utilities, but they are linear combinations of each other. In any case, the solution of this auxiliary problem will be dictatorial. Let $\check{z}^j \in Z^j$, $1 \leq j < i$, denote any solution to \eqref{eq:J3}. Then, \eqref{eq:J2} reduces to
\begin{align}
\label{eq:J4}
\sup_{y \in \Gamma(k)} \Big[U(f(k)-y,\bar{\theta})+\delta^i J(y,\bar{\theta})\Big] + \sum_{j < i} \bar{\theta}^j\delta^j \check{z}^j,
\end{align}
so the choice of $y$ at the margin is determined by $\delta^i$, whereas each of the $\bar{\theta}^j \delta^j \check{z}^j$ terms  works effectively as a nonnegative utility transfer to those agents with a higher level of patience than $i$. 

It is easy to show that at any period $t$, an optimal sharing rule $\tilde{x}_t^j$ must satisfy
\begin{align}
\label{eq:J5}
\left(\frac{\delta^i}{\delta^j}\right)^t =\frac{\bar{\theta}^j u_j'(\tilde{x}_t^j)}{\bar{\theta}^i u_i'(\tilde{x}_t^i)}, & & 1 \leq j < i;\quad t=0,1,\ldots
\end{align}
Hence as $t \to +\infty$, the ratio of marginal utilities $u_j'(\tilde{x}_t^j)/u_i'(\tilde{x}_t^i) \to 0$, which implies that  the consumption of each agent $j$ should grow faster than the consumption of $i$ in equilibrium. Note that \eqref{eq:J5} also  implies
\begin{align*}
\frac{u_i'(\tilde{x}_t^i)}{\delta^i u_i'(\tilde{x}_{t+1}^i)} =\frac{u_j'(\tilde{x}_t^j)}{\delta^j u_j'(\tilde{x}_{t+1}^j)}, & & 1 \leq j < i;\quad  t=0,1,\ldots
\end{align*}
At the same time, the evolution of aggregate variables is determined by the following Euler equation 
\begin{align*}
\tilde{U}'(\hat{x}_t)=\delta^i \,\tilde{U}'(\hat{x}_{t+1})\,f'(\bar{y}_t), & & t=0,1,\ldots,
\end{align*}
where $\bar{y}_t \in \Gamma(k_t)$ solves \eqref{eq:J4}, and the aggregate resource restriction is satisfied, then
\begin{align*}
\hat{x}_t = \sum_{j < i} \hat{x}_t^j + \hat{x}_t^i = f(k_t) - \bar{y}_t, & & t=0,1,\ldots
\end{align*}
But a program with such characteristics is clearly time inconsistent. 

To summarize, if the Pareto weights are restricted to be constant over time and the consumption-savings choice is endogenous, an optimal solution that is interior for all agents may be difficult (or even impossible) to achieve. In addition, the problems of time inconsistency and dictatorial preferences become interrelated and more complex in nature.

\section{Concluding remarks}
\label{sec:conclude}

It is not surprising that some strong assumptions, especially  time-additive preferences and interior optimal sharing rules, are needed to obtain the main results of this paper. Extending this framework to a broader class of intertemporal preferences is a natural direction for future work. It is also easily verified that, even within the ATCF class of utility functions, there exist conditions for which a null share for some agents may be optimal in some periods. Relaxing the assumption of an interior sharing rule would give rise to changes in the \emph{composition} of the group over time and richer dynamics could occur. In that case, the set of agents with a positive share would be included as a state variable of the system. This seems to be a challenging but promising direction for future research. 

\appendix

\section{Proofs}

\subsection{Proof of Theorem \ref{thm:RFBellman}}
\label{app:pfRFBellman}

Fix $k \in (0,k_m)$ and $\theta \in \Theta^n$ with $\theta^i > 0$, for all $i \in N$. Let $\lambda$ and $\mu$ be  nonnegative Lagrange multipliers associated to the inequality restrictions in \eqref{eq:PP}, and form the Lagrangean 
\begin{align}
\label{eq:longLagr}
\mathscr{L}(x,y,z,\tau,\lambda,\mu|\,k,\theta)=\textstyle\sum_i \theta^i[u_i(x^i) + \delta^i z^i]&+\lambda[f(k)-\textstyle\sum_i x^i-y]\\[3pt]
&-\mu[\textstyle\sum_i\tau^i z^i-V(y,\tau)].\nonumber
\end{align}
Denote by $X \subset \mathbb{R}_+^n$ the space of $x$ allocations, and let $\Phi := X \times \mathbb{R}_+\!\times \mathcal{U} \times \mathbb{R}_+$ and $\Upsilon:=\Theta^n \times \mathbb{R}_+$. The proof is divided into several steps. 

\begin{enumerate}[label=\textbf{Step  \arabic*.},wide=0\parindent]
\item Note that Proposition \ref{prp:Vfunct} implies that \eqref{eq:PP} is a saddle-point problem with a concave-convex objective function, defined on products of convex sets, that satisfies \emph{strong duality}.\footnote{For a recent treatment on strong duality in convex optimization problems, see \cite{bertsekasetal03}.} Hence 
\begin{equation}
\label{eq:strongdual}
V(k,\theta)=\adjustlimits\sup_{(x,y,z,\lambda) \in \Phi\,}\inf_{(\tau,\mu) \in \Upsilon}\ \mathscr{L}(x,y,z,\tau,\lambda,\mu|\,k,\theta),
\end{equation}
and the $\sup$ and $\inf$ above can be interchanged. By Lemma 1 in \cite{kamihigashi08}, the supremum over two variables can be split into two suprema. Then, the right-hand side of \eqref{eq:strongdual} is equivalent to:
\begin{align}
\label{eq:supdec1}
\sup_{x \in X\,}\adjustlimits\sup_{(y,z,\lambda) \in \Phi_{x}\,}\inf_{(\tau,\mu) \in \Upsilon} \big\{\left[\textstyle\sum_i \theta^i u_i(x^i) +\lambda(\hat{x}-\textstyle\sum_i x^i)\right] +\lambda[f(k)-\hat{x}-y]\big.\\
\big. + \textstyle\sum_i(\theta^i\delta^i-\mu\tau^i) z^i  + \mu V(y,\tau)\big\},\nonumber
\end{align}
where $\Phi_{x}:=  \{(y,z,\lambda) \in \mathbb{R}_+\!\times \mathcal{U} \times \mathbb{R}_+ : (x,y,z,\lambda) \in \Phi\}$.

\item Consider the following auxiliary program
\begin{equation*}
\sup_{x \in X}\ \left[\textstyle\sum_i \theta^i u_i(x^i) +\lambda\left(\hat{x}-\textstyle\sum_i x^i\right)\right],
\end{equation*}
for some $0 < \hat{x} \leq f(k)$ given. Since this problem is convex and differentiable, and the solution is interior for each $i$, Karush-Kuhn-Tucker (KKT) conditions are necessary and sufficient for optimality. Then, for each $i \in N$ there exists a map $\sigma^i: [0,1]\times \mathbb{R}_+\! \to \mathbb{R}_+$, such that 
\begin{align}
\label{eq:mguit}
\theta^i u_i'\left(\sigma^i(\theta^i,\lambda)\right)=\lambda,
\end{align}
and $\sum_i \sigma^i(\theta^i,\lambda) \leq \hat{x}$. The value of this auxiliary program is given by
\begin{equation}
\label{eq:Qdef}
Q(\lambda,\hat{x},\theta):=\textstyle\sum_i \theta^i
u_i\left(\sigma^i(\theta_i,\lambda)\right)+\lambda\left[\hat{x}-\textstyle\sum_i \sigma^i(\theta^i,\lambda)\right].
\end{equation} 
Next, replace $\sup_{x \in X}\left[\textstyle\sum_i \theta^i u_i(x^i)+\lambda(\hat{x}-\textstyle\sum_i x^i)\right]$ with $Q(\lambda,\hat{x},\theta)$ in \eqref{eq:supdec1}, and split the sup again, which implies that \eqref{eq:supdec1} is equivalent to
\begin{align*}
\sup_{z \in \mathcal{U}\,}\adjustlimits\sup_{(y,\lambda) \in \Phi_{x,z}}\inf_{(\tau,\mu) \in \Upsilon} \left\{Q(\lambda,\hat{x},\theta)+\lambda[f(k)\!-\!\hat{x}\!-\!y] + \left[\textstyle\sum_i(\theta^i\delta^i- \mu\tau^i)z^i\right] + \mu V(y,\tau)\right\},
\end{align*}
where $\Phi_{x,z}:=  \left\{(y,\lambda) \in \mathbb{R}_+\!\times\mathbb{R}_+: (y,z,\lambda) \in \Phi_x \right\}$.

\item Now, solve the following auxiliary program for $0 \leq y < f(k)$ given,
\begin{equation}
\label{eq:supdec2}
\adjustlimits\sup_{z \in \mathcal{U}\,}\inf_{(\tau,\mu) \in \Upsilon} \left[\textstyle\sum_i(\theta^i\delta^i- \mu\tau^i)z^i+\mu V(y,\tau)\right].
\end{equation}
If $\mu=0$, then each $i$ should get the sup of $z^i$ on $\mathcal{U}$, but this is infeasible. Assume $\mu > 0$, which implies $\sum_i \tau^i z^i = V(y,\tau)$ by complementary slackness. Differentiating the objective in \eqref{eq:supdec2} with respect to $\tau^i$ yields $z^i = \partial{V}/\partial{\tau^i}$, $i=1,\ldots,n$. 
On the other hand, given that $0 < \theta^i < 1$, optimality requires that $\tau^i$ satisfies 
\begin{align}
\label{eq:thetarect}
\theta^i\delta^i = \mu \tau^i,  
\end{align}
for each $i \in N$. This in turn implies $0 < \tau^i <1$ for all $i$, so adding up \eqref{eq:thetarect} over $i$ yields 
\begin{align*}
\mu=\sum_{i=1}^n \theta^i\delta^i,
\end{align*}
which is \eqref{eq:mudef}. Slightly abusing notation, call this map $\mu:\Theta^n \to \mathbb{R}_+$. Since \eqref{eq:thetarect} defines a transition map $F:\Theta^n \to \Theta^n$ for the utility weights, then $\tau = F(\theta)$ at an optimum. Taking all this into consideration, the right-hand side of \eqref{eq:strongdual} can be further simplified and written as
\begin{align}
\label{eq:Bellmanint}
\sup_{(y,\lambda) \in \Phi_{x,z}} \left\{Q(\lambda,\hat{x},\theta)+\lambda[f(k)-\hat{x}-y] + \mu(\theta)V(y,F(\theta))\right\}.
\end{align}

\item By \eqref{eq:mguit}, individual allocations depend only on $(\theta,\lambda)$, so an optimal choice of $x$ can be formulated in terms of aggregate consumption $\hat{x}$, instead of $\lambda$. Given that each $u_i$ is strictly increasing and $\lambda > 0$ holds at the optimum, set $\hat{x}=\sum_i \sigma^i(\theta^i,\lambda)$ in \eqref{eq:Qdef} and use this expression to implicitly define a function $\lambda: \hat{X}\!\times\Theta^n \to \mathbb{R}_+$. This, in turn, allows to define an optimal consumption profile as a map $s:\hat{X} \times \Theta^n \to X$ by setting $s^i(\hat{x},\theta):=\sigma^i(\theta^i,\lambda(\hat{x},\theta))$. Given that the value of \eqref{eq:Qdef} becomes
\begin{equation*}
Q(\lambda(\hat{x},\theta),\hat{x},\theta)=\sum_{i=1}^n \theta^i u_i\left(\sigma^i(\theta^i,\lambda(\hat{x},\theta))\right),
\end{equation*}
it is possible to define a map $U : \hat{X} \times \Theta^n \to \underline{\mathbb{R}}$ by substituting $\sigma^i$ with $s^i$ on the right-hand side of the above expression and setting $U(\hat{x},\theta):=Q(\lambda(\hat{x},\theta),\hat{x},\theta)$. This gives \eqref{eq:Udef}.

\item Finally, substituting $U(\hat{x},\theta)$ into \eqref{eq:Bellmanint} yields
\begin{align*}
V(k,\theta)= &\sup_{(\hat{x},y) \gg 0}\ \Big[U(\hat{x},\theta)+\mu(\theta)V(y,F(\theta)):\hat{x} + y \leq f(k) \Big],
\end{align*}
that is equivalent to \eqref{eq:strongdual}. It is readily verified that $U(\hat{x},\theta)$ is strictly increasing in $\hat{x}$ for all $\theta$. To see this, implicitly differentiate $\sigma^i$ with respect to $\lambda$ in \eqref{eq:mguit} to obtain $\partial{\sigma^i}/\partial{\lambda}=(\theta_i u_i''(\sigma^i))^{-1} < 0$, for each $i$. It also follows from implicit differentiation that $\partial{\lambda}/\partial{\hat{x}}=\sum_i \partial{\sigma^i}/\partial{\lambda} < 0$. Therefore,
\begin{equation*}
\frac{\partial{U}}{\partial{\hat{x}}}= \sum_{i=1}^n \theta^i u_i'(\sigma^i)\frac{\partial{\sigma^i}}{\partial{\lambda}}\frac{\partial{\lambda}}{\partial{\hat{x}}}=\sum_{i=1}^n  \lambda\,\frac{\partial{\sigma^i}}{\partial{\lambda}}\frac{\partial{\lambda}}{\partial{\hat{x}}} > 0.
\end{equation*}
This, together with the strict monotonicity of $V$ with respect to its first argument implies that the resource  restriction, $\hat{x} \leq f(k)-y$, holds with equality, which gives \eqref{eq:RFBellman}. 
\end{enumerate}

Since $(k,\theta)$ was chosen arbitrarily in the interior of $K \!\times\Theta^n$, all the above results can be extended to that set. This completes the proof. \qed

\subsection{Proof of Proposition \ref{prp:Uchar}}

The proofs for parts \ref{prp:Umonccv}, \ref{prp:Udiffb}, and \ref{prp:sdiffb} are omitted. The interested reader is referred to the extensive work of \cite{mascolell89} for details. In order to simplify the proofs for the remaining results, it will be assumed that the  differentiability of $U$ and $s$ has already been established. 

Fix $\hat{x} \in (0,+\infty)$ and suppose that $\theta^i >0$ holds for all $i \in N$. Differentiating \eqref{eq:resrecs}  with respect to $\theta^i$ yields
\begin{align}
\label{eq:partsi}
\sum_{j=1}^n \frac{\partial{s^j(\hat{x},\theta)}}{\partial{\theta^i}}=0. 
\end{align}
Next, differentiate $U(\hat{x},\theta)=\sum_i \theta^i u_i(s^i(\hat{x},\theta))$ with respect to $\theta^i$, so that
\begin{align*}
\frac{\partial{U(\hat{x},\theta)}}{\partial{\theta^i}}
&=u_i(s^i(\hat{x},\theta))+ \sum_{j=1}^n \theta^j u_j'(s^j(\hat{x},\theta))\frac{\partial{s^j(\hat{x},\theta)}}{\partial{\theta^i}}\\
&=u_i(s^i(\hat{x},\theta))+ \lambda(\hat{x},\theta)\sum_{j=1}^n \frac{\partial{s^j(\hat{x},\theta)}}{\partial{\theta^i}},
\end{align*}
where the second equality uses the fact that $\theta^i u_i'(s^i(\hat{x},\theta))=\lambda(\hat{x},\theta)$, from the KKT optimality conditions. By \eqref{eq:partsi}, the above equality reduces to
\begin{align}
\label{eq:partUtheta}
\frac{\partial{U(\hat{x},\theta)}}{\partial{\theta^i}}
= u_i(s^i(\hat{x},\theta)), \qquad \text{for all}\ i \in N,
\end{align}
which in turn implies that
\begin{align*}
\sum_{i=1}^n \theta^i\frac{\partial{U(\hat{x},\theta)}}{\partial{\theta^i}} = \sum_{i=1}^n \theta^i u_i(s^i(\hat{x},\theta))=U(\hat{x},\theta),
\end{align*}
hence, for every $\hat{x} \in \hat{X}$, the aggregate utility function $U(\hat{x},\theta)$ is homogeneous of degree one in $\theta \in \Theta^n$ by Euler's theorem for homogeneous functions. This proves \ref{prp:Uhomog}.

For part \ref{prp:shomog}, since $s^i$ is nonnegative for each $i \in N$, it is clear from \eqref{eq:resrecs} that $\hat{x}=0$ must imply $s^i(0,\theta)=0$, for all $\theta \in \Theta^n$. To prove that $s^i(\cdot,\theta)$ is homogeneous of degree zero, by a corollary from Euler's theorem for homogeneous functions, it is straightforward to show that if $U(\cdot,\theta)$ is homogeneous of degree one, its partial derivatives are homogeneous of degree zero. Differentiating \eqref{eq:partUtheta} with respect to $\theta^j$, $j=1,\ldots,n$, gives
\begin{align*}
 \frac{\partial}{\partial{\theta^j}}\left[\frac{\partial{U(\hat{x},\theta)}}{\partial{\theta^i}}\right]=u_i'(s^i(\hat{x},\theta))\frac{\partial{s^i(\hat{x},\theta)}}{\partial{\theta^j}},
\end{align*}
hence
\begin{align*}
\sum_{j=1}^n \theta^j \frac{\partial}{\partial{\theta^j}}\left[\frac{\partial{U(\hat{x},\theta)}}{\partial{\theta^i}}\right]
=\frac{\lambda(\hat{x},\theta)}{\theta^i}\left[\sum_{j=1}^n \theta^j \frac{\partial{s^i(\hat{x},\theta)}}{\partial{\theta^j}}\right]=0.
\end{align*}
But $\lambda(\hat{x},\theta)>0$ in any interior equilibrium and $\theta^i > 0$ by hypothesis, so the sum in square brackets in the second equality must vanish, yielding the desired result. \qed

\subsection{Proof of Lemma \ref{lem:betadiff}}

Let $\Delta{t}$, $\Delta\tau$ be nonnegative integers such that $t'=t+\Delta{t}$ and $\tau'=\tau+\Delta\tau$. Then, the sign of 
$\beta_{t+\tau}/\beta_{t+\tau'}-\beta_{t'+\tau}/\beta_{t'+\tau'}$
is the same as 
\begin{align*}
\left[\sum_{i=1}^n \theta_0^i(\delta^i)^{t+\tau}\right]\left[\sum_{j=1}^n \theta_0^j(\delta^j)^{t+\Delta{t}+\tau+\Delta{\tau}}\right]-\left[\sum_{i=1}^n \theta_0^i(\delta^i)^{t+\tau+\Delta{\tau}}\right]\left[\sum_{j=1}^n \theta_0^j(\delta^j)^{t+\Delta{t}+\tau}\right].
\end{align*}
After expanding both products in this expression, it is easy to see that all factors with $i=j$ cancel each other out. Collecting the remaining terms yields
\begin{align*}
\sum_{\substack{i,j=1\\[1pt] i < j}}^n \left[\theta_0^i(\delta^i)^{t+\tau}\right]\left[\theta_0^j(\delta^j)^{t+\tau}\right]\left[(\delta^i)^{\Delta{t}}-(\delta^j)^{\Delta{t}}\right]\left[(\delta^i)^{\Delta\tau}-(\delta^j)^{\Delta\tau}\right].
\end{align*}
Note that if either $\Delta{t}$ or $\Delta{\tau}$ is zero, the above sum vanishes. Otherwise, provided that $\Delta{t} > 0$ and $\Delta{\tau} > 0$, it follows from hypothesis that $\delta^i - \delta^j > 0$ for at least one pair $i,j$ such that $i > j$. Hence the sum must be strictly positive. This completes the proof. \qed

\bibliographystyle{chicago}
\bibliography{TCCP}

\end{document}